\documentclass[12pt, reqno]{amsart}
\usepackage{amsmath, amsthm, amscd, amsfonts, amssymb, graphicx, color}
\usepackage[bookmarksnumbered, colorlinks, plainpages]{hyperref}
\hypersetup{colorlinks=true,linkcolor=red, anchorcolor=green, citecolor=cyan, urlcolor=red, filecolor=magenta, pdftoolbar=true}

\input{mathrsfs.sty}

\newtheorem{theorem}{Theorem}[section]
\newtheorem{lemma}[theorem]{Lemma}

\newtheorem{corollary}[theorem]{Corollary}
\theoremstyle{definition}

\newtheorem{example}[theorem]{Example}

\theoremstyle{remark}
\newtheorem{remark}[theorem]{Remark}
\numberwithin{equation}{section}

\begin{document}

\title[Operator equations in Hilbert $C^*$-modules]{Operator equations $AX+YB=C$ and $AXA^*+BYB^*=C$ in Hilbert $C^*$-modules}

\author[Z. Mousavi]{Z. Mousavi}
\address[Z. Mousavi]{ Department of Engineering, Abhar Branch, Islamic Azad University, Abhar, Iran.}
\email{mousavi.z@abhariau.ac.ir}

\author[R. Eskandari]{R. Eskandari }
\address[R. Eskandari]{Department of Mathematics, Faculty of Science, Farhangian University, Tehran, Iran.}
\email{eskandarirasoul@yahoo.com}

\author[M.S. Moslehian]{M. S. Moslehian}
\address[M.S. Moslehian]{Department of Pure Mathematics, Center of Excellence in
Analysis on Algebraic Structures (CEAAS), Ferdowsi University of Mashhad, P. O. Box 1159, Mashhad 91775, Iran.}
\email{moslehian@um.ac.ir and moslehian@member.ams.org}

\author[F. Mirzapour]{F. Mirzapour }
\address[F. Mirzapour]{ Department of Mathematics, Faculty of Sciences, University of Zanjan, P.O. Box 45195-313, Zanjan, Iran.}
\email{f.mirza@znu.ac.ir}



\begin{abstract}
Let $A,B$ and $C$ be adjointable operators on a Hilbert $C^*$-module $\mathscr{E}$. Giving a suitable version of the celebrated Douglas theorem in the context of Hilbert $C^*$-modules, we present the general solution of the equation $AX+YB=C$ when the ranges of $A,B$ and $C$ are not necessarily closed. We examine a result of Fillmore and Williams in the setting of Hilbert $C^*$-modules. Moreover, we obtain some necessary and sufficient conditions for existence of a solution for $AXA^*+BYB^*=C$. Finally, we deduce that there exist nonzero operators $X, Y\geq 0$ and $Z$ such that $AXA^*+BYB^*=CZ$, when $A, B$ and $C$ are given subject to some conditions.
\end{abstract}

\maketitle

\section{Introduction and Preliminaries}
Recently several operator equations have been extended from matrices to infinite dimensional spaces, i.e., Hilbert spaces and Hilbert $C^*$-modules; see \cite{MMM} and references therein.
Recall that the notion of Hilbert $C^*$-module is a natural generalization of that of Hilbert space arising by replacing the field of scalars $\mathbb{C}$ by a $C^*$-algebra.\\
Generalized inverses are useful tools for investigation of solutions of operator equations in the setting of Hilbert $C^*$-modules but these inverses need the strong condition of closedness of ranges of considered operators. Fang et al. \cite{FangYuYao, FangYu} have studied the solvability of operator equations without the closedness condition on ranges of operators by employing a generalization of a known theorem of Douglas \cite[Theorem 1]{Douglas} in the framework of Hilbert $C^*$-modules. In their results, concentration is based on the idea of using more general (orthogonal) projections instead of projections such as $AA^{\dag}$. They investigated the equations $AX=B$ \cite[Theorem 1.1]{FangYuYao}, $A^*XB+B^*X^*A=C$ \cite[Corollary 2.8]{FangYu} and $AXB=C$ \cite[Theorem 3.4]{FangYu}.

Inspired by Fang et al., we investigate the solution of equations $AX+YB=C$ and $AXA^*+BYB^*=C$ without the condition of closedness of ranges.
This paper is organized as follows. First, we recall some basic information about Hilbert $C^*$-modules. In Section 2,
we present an example that shows that the conditions $CC^*\leq \lambda AA^*$ for some $\lambda > 0$ and $\mathcal{R}(C)\subseteq \mathcal{R}(A)$ are not equivalent in the setting of Hilbert $C^*$-modules, in general. It shows that Theorem 1.1 of \cite{FangYuYao} is not true in the current form. Instead, we present a suitable revision of it and provide some variants of it. The equation $AX+YB=C$ is studied in Section 3 and we present a general solution for it. In addition, we slightly extend a result of Fillmore and Williams \cite{FillmoreWilliams} to Hilbert $C^*$-modules and establish some necessary and sufficient conditions for existence of a solution for $AXA^*+BYB^*=C$. Finally, we show under some conditions that there exist nonzero (nontrivial) positive operators $X,~Y$ and a nonzero operator $Z$ such that $AXA^*+BYB^*=CZ$.

Throughout the paper, $\mathscr{A}$ denotes a $C^{*}$-algebra. A Hilbert $C^*$-module is a right $\mathscr{A}$-module equipped with an $\mathscr{A}$-valued inner product $\langle \cdot, \cdot \rangle: \mathscr{E} \times \mathscr{E} \rightarrow \mathscr{A}$ such that the induced norm $\| x\| = \| \langle x,x\rangle \|^{\frac{1}{2}}$ is complete. An inner-product $\mathscr{A}$-module $\mathscr{E}$ has an ``$\mathscr{A}$-valued norm" $|\cdot|$, defined by $|x|=\langle x,x\rangle^{\frac{1}{2}}$. Suppose that $\mathscr{E}$ and $\mathscr{F}$ are Hilbert $C^{*}$-modules. Let $\mathcal{L}(\mathscr{E},\mathscr{F})$ be the set of all bounded $\mathscr{A}$-linear maps $A: \mathscr{E}\rightarrow \mathscr{F}$ for which there is a map $A^{*}: \mathscr{F}\rightarrow \mathscr{E}$, the adjoint of $A$, such that $\langle Ax,y\rangle= \langle x,A^{*}y\rangle \,\,(x\in \mathscr{E}, y\in \mathscr{F})$. We use $\mathcal{L}(\mathscr{E})$ to denote the $C^*$-algebra $\mathcal{L}(\mathscr{E},\mathscr{E})$.
If $ A\in \mathcal{L}(\mathscr{E})$ has closed range, then the Moore-Penrose inverse of $A$ (\cite{XuSheng}), denoted by $A^\dag$, is defined as the unique element of $\mathcal{L}(\mathscr{E})$ such that
$$
AA^{\dag}A=A, \quad A^{\dag}AA^{\dag}=A^{\dag},\quad (A^{\dag}A)^*=A^{\dag}A,\quad (AA^{\dag})^*=AA^{\dag}.
$$

We use $\mathcal{R}(A)$ and $\mathcal{N}(A)$ for the range and the null space of an operator $A$, respectively. We say that a closed submodule $\mathscr{F}$ of a Hilbert $C^*$-module $\mathscr{E}$ is orthogonally complemented if $\mathscr{E}=\mathscr{F}\oplus \mathscr{F}^{\perp}$, where $\mathscr{F}^{\perp}=\{ x\in \mathscr{E} : \langle x,y\rangle=0$ for all $ y\in \mathscr{F}\}$. Evidently a Hilbert $C^*$-submodule $\mathscr{F}$ of a Hilbert $C^*$-module $\mathscr{E}$ is orthogonally complemented if and only if there exists a projection $P$ on $\mathscr{E}$, whose range is $\mathscr{F}$ and $\mathcal{R}(P)\oplus \mathcal{N}(P)=\mathscr{E}$. The following theorem is known.
\begin{theorem}\label{t21}
 $($\cite[Theorem 3.2]{Lance}$)$. Let $A\in \mathcal{L}(\mathscr{E},\mathscr{F})$. Then
\begin{enumerate}
\item $\mathcal{R}(A)$ is closed if and only if $\mathcal{R}(A^*)$ is closed, and in this case, $\mathcal{R}(A)$ and
 $\mathcal{R}(A^*)$ are orthogonally complemented with $\mathcal{R}(A)=\mathcal{N}(A^*)^{\perp}$ and
 $\mathcal{R}(A^*)=\mathcal{N}(A)^{\perp}$.
\item $\mathcal{N}(A)=\mathcal{N}(|A|)$, $\mathcal{N}(A^*)=\mathcal{R}(A)^{\perp}$ and
 $\mathcal{N}(A^*)^{\perp}=\mathcal{R}(A)^{\perp \perp}\supseteq \overline{\mathcal{R}(A)}.$
\end{enumerate}
\end{theorem}

 For any $A\in \mathcal{L}(\mathscr{E},\mathscr{F})$, if $\overline{\mathcal{R}(A^*)}$ is orthogonally complemented, then $\overline{\mathcal{R}(A^*)}^{\perp}=\mathcal{N}(A)$ and the orthogonal decomposition $\mathscr{E}=\overline{\mathcal{R}(A^*)} \oplus \mathcal{N}(A)$ is valid. Let $P_{A^*}$ be the projection of $\mathscr{E}$ onto $\overline{\mathcal{R}(A^*)}$. Then $P_{A^*}A^*=A^*$ and $N_{A}A^*=0$, where $N_{A}=I-P_{A^*}$, where $I$ denotes the identity operator on $\mathcal{E}$; cf. \cite[p.~21]{Lance}.

 The so-called ``reduced solution'' that is defined in the following Theorem will be used throughout the paper.

 \begin{theorem}\label{t23}
 $($\cite[Theorem 3.1]{FangYu}$)$. Let $A\in \mathcal{L}(\mathscr{E},\mathscr{F})$, $B\in \mathcal{L}(\mathscr{G},\mathscr{H})$ and $C\in \mathcal{L}(\mathscr{G},\mathscr{F})$.
 Suppose that $\overline{\mathcal{R}(A^*)}$ and $\overline{\mathcal{R}(B)}$ are orthogonally complemented.
If
 $$\mathcal{R}(C)\subseteq \mathcal{R}(A)\quad and \quad \overline{\mathcal{R}(C^*)}\subseteq \mathcal{R}(B^*)$$
$$(or \quad \overline{\mathcal{R}(C)}\subseteq \mathcal{R}(A)\quad and \quad\mathcal{R}(C^*)\subseteq \mathcal{R}(B^*)),$$
then $AXB=C$ has a unique solution $D\in \mathcal{L}(\mathscr{H},\mathscr{E})$ such that
$\mathcal{R}(D)\subseteq \mathcal{N}(A)^{\perp}$ and $\mathcal{R}(D^*)\subseteq \mathcal{N}(B^*)^{\perp}$,
which is called the reduced solution, and the general solution to $AXB=C$ is of the form
$$X=D+N_AV_1+V_2N_{B^*},$$ where $V_1,V_2 \in \mathcal{L}(\mathscr{H},\mathscr{E})$.
\end{theorem}

\section{On a result of Fang et al.}
The condition of having an orthogonal complement for the closure of a range is weaker than that of having a closed range through solving operator equations. The following example gives an operator acting on a Hilbert $C^*$-module, which does not have closed range but the closure of its range is orthogonally complemented.

\begin{example}\label{ex1}
Let $H$ be a Hilbert space and $\mathscr{A}$ be a unital $C^*$-algebra. Recall that the Hilbert $\mathscr{A}$-module $\oplus\mathscr{A}_i$, where each $\mathscr{A}_i$ is a copy of $\mathscr{A}$ consists of all $(a_i)\in \oplus\mathscr{A}_i $ such that $\sum \langle a_i,a_i\rangle$ is norm-convergent in $\mathscr{A}$. Suppose that
$T: \oplus\mathscr{A}_i\longrightarrow \oplus\mathscr{A}_i$ is defined by $(a_i) \mapsto (0,\frac{1}{1}a_2,0,\frac{1}{2}a_4,0, \cdots )$. Let $1$ be the unit element of $C^*$-algebra $\mathscr{A}$ and $x_{2n}=(1,1,\cdots,1,0,0, \cdots)$, whose first $2n$ entries are 1 and the rest are $0$. We have
 $$T(x_{2n})=(0,1,0,\frac{1}{2}, \cdots,\frac{1}{n},0,0, \cdots)\in \mathcal{R}(T).$$
Letting $n\rightarrow\infty$, we get $T(x_{2n})\rightarrow (0,1,0,\frac{1}{2},0,\frac{1}{3},0, \cdots)$, which does not belong to $\mathcal{R}(T)$ because if there exists an element $z=(b_i)\in \oplus\mathscr{A}_i$ such that $T(z)=(0,1,0,\frac{1}{2},0,\frac{1}{3},0, \cdots)$, then $b_i=1$ for all $i$, which gives rise to a contradiction due to $\sum\langle1,1\rangle$ is not convergent. Thus $\mathcal{R}(T)$ is not closed.

Next, we show that $\overline{\mathcal{R}(T)}$ is orthogonally complemented. For any $(c_i)\in \oplus\mathscr{A}_i $, we have
$$ (c_i)=(c_1,0,c_3,0, \cdots)+(0,c_2,0,c_4,0, \cdots),$$ where $(c_1,0,c_3,0, \cdots) \in \overline{\mathcal{R}(T)}^{\perp}$. Finally, it is sufficient to put $y_{2n}=(0,c_2,0,2c_4,0, $\\$\cdots,nc_{2n},0,0, \cdots)$ and note that $T(y_{2n})\rightarrow (0,c_2,0,c_4,0, \cdots)\in \overline{\mathcal{R}(T)}$.
Thus we have an operator acting on a Hilbert $C^*$-module, which does not have closed range but the closure of its range is orthogonally complemented.
\end{example}

We remind the following result of Fang et al. and observe that this result needs some revision.

{\bf \cite[Theorem 1.1]{FangYuYao}.}
\emph{Let $C\in \mathcal{L}(\mathscr{G},\mathscr{F})$, $A\in \mathcal{L}(\mathscr{E},\mathscr{F})$ and $\overline{\mathcal{R}(A^*)}$
be orthogonally complemented. The following statements are equivalent:
\begin{enumerate}
 \item $CC^*\leq \lambda AA^*$ for some $\lambda > 0;$
 \item There exists $\mu > 0$ such that $\| C^*z\|\leq \mu \| A^*z\|$ for any $z\in \mathscr{F}$;
 \item There exists $D\in \mathcal{L}(\mathscr{G},\mathscr{E})$ such that $AD=C$, i.e., $AX=C$ has a solution;
 \item $\mathcal{R}(C)\subseteq \mathcal{R}(A)$.
\end{enumerate}
In this case, there exists a unique operator $X$ satisfying $ \mathcal{R}(X)\subseteq \mathcal{N}(A)^{\perp}$, which is called
the reduced solution and is denoted by $D$ and defined as follows:
\begin{equation}\label{e21}
D=P_{A^*} A^{-1}C\,,\quad D^*\mid_{\mathcal{N}(A)}=0\,,\quad D^{*}A^{*}y=C^*y\quad (y\in \mathscr{F}),
\end{equation}
where $A^{-1}$ does not refer to the inverse of $A$ but to the expression of inverse image.}

Carefully checking the proof of \cite[Theorem 1.1]{FangYuYao} shows that there is a gap in it. Let $\mathscr{H}$ be a Hilbert $C^*$-module. Let $A\in \mathcal{L}(\mathscr{H})$ be given such that $\overline{\mathcal{R}(A^*)}$ is orthogonally complemented. The authors of \cite{FangYuYao} proved $(2)\Rightarrow (3)$ by introducing
$D$ as (\ref{e21}), where the notation $A^{-1}$ stands for the inverse image rather than the inverse of $A$. But if $ \mathcal{R}(C)$ is not contained in $\mathcal{R}(A)$, then $A^{-1}C$ is meaningless, which leads to the wrong way to well define $D$ as above. The same problem appears in the proof of $(2)\Rightarrow (4)$ in Theorem 1.1 of \cite{FangYuYao}. It can be deduced from the proof of Theorem 1.1 of \cite{FangYuYao} that definitely

\begin{itemize}
\item Item (4) $\Leftrightarrow$ Item (3);
\item Item (3)$\Rightarrow $ Item (1);
\item Item (1)$\Rightarrow$ Item (2).
\end{itemize}
The next example shows that (2) and (4) are not equivalent.

\begin{example}
Let $\mathscr{A}=C[0,1]$ and let $\mathscr{M}=\{f\in \mathscr{A}, f(0)=0\}$ be its maximal ideal regarded as a Hilbert $\mathscr{A}$-module. We define $A\in\mathcal{L}(\mathscr{M},\mathscr{A})$ and $C\in\mathcal{L}(\mathscr{A})$ by
\[
(Af)(\lambda)=\lambda f(\lambda)\qquad (f \in \mathscr{M})
\]
and
\[
(Cf)(\lambda)=\lambda f(\lambda) \qquad (f\in \mathscr{A}).
\]
We have $C^*=C$ and $(A^*f)(\lambda)=\lambda f(\lambda)\,\, (f \in \mathscr{A})$. Clearly $\|C^*(f)\|\leq \|A^*(f)\|$ but $\mathcal{R}(C)\not\subseteq\mathcal{R}(A)$, since for $f_0(\lambda)=\lambda$ we have $f_0\in \mathcal{R}(C)\setminus \mathcal{R}(A)$. Now we show that $\overline{\mathcal{R}(A^*)}$ is orthogonally complemented. It is enough to show that $\mathcal{R}(A^*)$ is dense in $\mathscr{M}$: Let $f\in\mathscr{M}$, and define $f_n\in\mathscr{A}$ by
\[f_n(\lambda)=\begin{cases}\frac{f(\lambda)}{\lambda}&\lambda\in[\frac{1}{n},1]\\
nf(\frac{1}{n})&\lambda\in [0,\frac{1}{n}]\end{cases}
\]
 Then
 \begin{align*}
 \|A^*f_n-f\|&=\sup_{\lambda\in[0,1]}\left|A^*f_n(\lambda)-f(\lambda)\right|\\
 &=\sup_{\lambda\in[0\frac{1}{n}]}\left|n\lambda f\left(\frac{1}{n}\right)-f(\lambda)\right|\\
 &\leq \sup_{\lambda\in[0\frac{1}{n}]} \left|n\lambda f\left(\frac{1}{n}\right)\right|+\sup_{\lambda\in[0\frac{1}{n}]} \left|f(\lambda)\right|\\
 &\leq \left|f\left(\frac{1}{n}\right)\right|+\left|f(\lambda_n)\right|\qquad (\text{~for some~} \lambda_n\in [0,1/n])
 \end{align*}
It follows from $\lim_nf\left(\frac{1}{n}\right)=\lim_nf(\lambda_n)=f(0)=0$ that $\|A^*f_n-f\|\to 0$ as $n \to \infty$ uniformly on $[0,1]$. Hence $A^*f_n \to f$ as $n \to \infty$ in $\mathscr{M}$.
 \end{example}

Now, we restate \cite[Theorem 1.1]{FangYuYao} in another form.\\

\begin{theorem}\label{t22}
 Let $C\in \mathcal{L}(\mathscr{G},\mathscr{F})$, $A\in \mathcal{L}(\mathscr{E},\mathscr{F})$. Suppose that $\overline{\mathcal{R}(A^*)}$
is orthogonally complemented. The following statements are equivalent:
\begin{enumerate}
 \item $\mathcal{R}(C)\subseteq \mathcal{R}(A)$;
 \item There exists $D\in \mathcal{L}(\mathscr{G},\mathscr{E})$ such that $AD=C$, i.e., $AX=C$ has a solution.
 \end{enumerate}
And if one of above statements holds, then
 $$CC^*\leq \lambda AA^* ~~for~ some~~ \lambda > 0.$$
In this case, there exists a unique operator $X$ satisfying $ \mathcal{R}(X)\subseteq \mathcal{N}(A)^{\perp}$, which is called the reduced
solution, denoted by $D$, and is defined as follows:
$$
D=P_{A^*} A^{-1}C\,,\quad D^*\mid_{\mathcal{N}(A)}=0\,,\quad D^{*}A^{*}y=C^*y\quad (y\in \mathscr{F}),
$$
where $A^{-1}$ does not refer to the inverse of $A$ but is the expression of inverse image.
\end{theorem}

This gap encourages us to find a solution for the operator equation $AX=C$ under other conditions. We would state the following theorem.

\begin{theorem}
Let $A, C\in \mathcal{L}(\mathscr{E})$. Suppose that $\overline{\mathcal{R}(A^*)}$ and $\overline{\mathcal{R}(C^*)}$
 are orthogonally complemented in $\mathscr{E}$. If $CC^*=\lambda AA^*$ for some $\lambda >0$, then $AX=C$ has a solution.
\end{theorem}
\begin{proof}
Let $D: \mathcal{R}(A^*)\rightarrow \mathcal{R}(C^*)$ be defined by $D(A^*z)=C^*z\,\,(z\in \mathscr{E})$. By assumption, we have $CC^*=\lambda AA^*$, so
$$ \|C^*z\|= \lambda ^{\frac{1}{2}}\|A^*z\|\quad \text{for ~any ~}z\in \mathscr{E}.$$
Hence $D$ is a well-defined bounded linear operator. So we can extend $D$
to $\overline{\mathcal{R}(A^*)}$ which is denoted by the same $D$. Then
\begin{equation}
 \tilde{D}(y)=\left\{
\begin{array}{lr}
D(y) \qquad\qquad ~~~~~y\in \overline{\mathcal{R}(A^*)};\\
0 \qquad\qquad\qquad ~~~~~~~~~~~~y \in \mathcal{N}(A),\\
\end{array} \right.
\end{equation}
is well-defined and $\tilde{D}A^*=C^*$. Similarly, there is a bounded linear map $\tilde{D^\prime}$ vanishing on $\mathcal{N}(C)$ such that $\tilde{D^\prime}C^*=A^*$. It follows from
$$\langle \tilde{D}A^*x, C^*y\rangle=\langle C^*x,C^*y\rangle=\langle x, CC^*y\rangle=\langle x, \lambda AA^*y\rangle=\langle A^*x,\lambda \tilde{D^\prime}C^*y\rangle $$
that $\tilde{D}$ is adjointable and $\tilde{D}^*=\lambda\tilde{D^\prime}$. Hence $A\tilde{D}^*=C$.
\end{proof}

By using the above theorem we deduce the following result.
\begin{corollary}\label{c31}
Let $A, C\in \mathcal{L}(\mathscr{E})$. Suppose that $\overline{\mathcal{R}(A^*)}$ and $\overline{\mathcal{R}(C^*)}$
 are orthogonally complemented in $\mathscr{E}$. If $AA^*=\lambda CC^*$ for some $\lambda >0$, then $\mathcal{R}(A)=\mathcal{R}(C)$.
\end{corollary}

\begin{corollary}\label{c32}
Let $T\in \mathcal{L}(\mathscr{E})$. Suppose that $\overline{\mathcal{R}(T^*)}$ and $\overline{\mathcal{R}(|T^*|)}$
 are orthogonally complemented in $\mathscr{E}$. Then $\mathcal{R}(T)=\mathcal{R}(|T^*|)$.
\end{corollary}
\begin{proof}
It is sufficient to put $A=T$, $C=|T^*|$ and $\lambda =1$ in Corollary \ref{c31}.
\end{proof}

\section{ Solutions to $AX+YB=C$ and $AXA^*+BYB^*=C$}

 The matrix equations $AX-YB=C$ and $AXB-CYD=E$ are studied by Baksalary and Kala in \cite{BaksalaryKala1} and \cite{BaksalaryKala2}, respectively,
 in which generalized inverses have been used to obtain some necessary and sufficient conditions for the existence of a solution.
 It is well-known that the matrix equation $AX+YB=C$ has a solution if and only if
\begin{equation}\label{Q1}
(I-AA^{\dag})C(I-B^{\dag}B)=0\,,
\end{equation}
and the general solution is of the form
$$
X=A^{\dag}C+A^{\dag}ZB+(I-AA^{\dag})W, \quad Y=(I-AA^{\dag})CB^{\dag}+(I-AA^{\dag})ZBB^{\dag}-Z\,,
$$
where $Z$ and $W$ are arbitrary matrices \cite{BaksalaryKala1}.

In the following, we shall investigate the solution of above equations with adjointable operators in the setting of Hilbert $C^*$-modules when neither $\mathcal{R}(A)$ nor $\mathcal{R}(B)$ is necessarily closed. First, we consider the equation $AX+YB=0$.
\begin{theorem}\label{t31}
Let $A, B\in \mathcal{L}(\mathscr{E})$. Suppose that $\overline{\mathcal{R}(A^*)}$, $\overline{\mathcal{R}(A)}$, $\overline{\mathcal{R}(B^*)}$ and $\overline{\mathcal{R}(B)}$ are orthogonally
complemented. Then the operators
\begin{eqnarray}\label{xhyh}
X_{h}=N_{A}W_{1}+W_2P_{B^*},\quad Y_{h}=P_AW_3+W_{4}N_{B^*},
\end{eqnarray}
satisfy the equation $AX+YB=0$, where $W_{1},W_{2}$, $W_3$ and $W_{4}$ are arbitrary operators in $\mathcal{\mathcal{L}}(\mathscr{E})$ such that
$AW_2P_{B^*}+P_AW_3B=0$. Moreover, any solution of $AX+YB=0$ is of the form \eqref{xhyh}.
\end{theorem}
\begin{proof}
Since
\begin{align*}
A(N_{A}W_1+W_2P_{B^*})&+(P_AW_3+W_4N_{B^*})B\\
&=AN_{A}W_1+AW_2P_{B^*}+P_AW_3B+W_4N_{B^*}B\\
&=0,
\end{align*}
so $X_h$ and $Y_h$ satisfy Equation \eqref{xhyh}.
Now we assume that there exist $X$ and $Y$ such that $AX+YB=0$. So we have $AXN_B=0$. By Theorem \ref{t23}, there exist $W_1$ and $W_2$ such that $X=N_{A}W_{1}+W_2N_{N_B}$. Note that $P_{N_B}$ is the projection on $\overline{\mathcal{R}(N_B)}$ and $\overline{\mathcal{R}(N_B)}=\overline{\mathcal{R}(I-P_{B^*})}=\overline{\mathcal{R}(B^*)}^\perp$, so
 $N_{N_B}=I-P_{N_B}=P_{B^*}$. Thus $X=N_AW_1+W_2P_{B^*}$.\\
 Similarly, $N_{A^*}YB=0$ has a solution. Hence, there exist $W_3$ and $W_4$ such that $Y=N_{N_{A^*}}W_3+W_{4}N_{B^*}=P_AW_3+W_4N_{B^*}$.
\end{proof}
\begin{remark}\label{r31}
Note that it is easy to find a suitable form for $W_2$ and $W_3$. For example, if we put $W_2=-P_{A^*}W'B$ and $W_3=AW'P_B$, where $W'$ is an arbitrary
operator, then we have $AW_2P_{B^*}+P_AW_3B=-AP_{A^*}W'BP_{B^*}+P_AAW'P_BB=-AW'B+AW'B=0$.
\end{remark}

We are ready to state our first main result regarding the following equation
 \begin{equation}\label{e31}
 AX+YB=C.
 \end{equation}

\begin{theorem}\label{t32}
Let $A, B, C\in \mathcal{L}(\mathscr{E})$. Suppose $\overline{\mathcal{R}(A)}$, $\overline{\mathcal{R}(B^*)}$, $\overline{\mathcal{R}(A^*)}$ and $\overline{\mathcal{R}(B)}$ are orthogonally complemented. If
$\mathcal{R}(CN_B)\subseteq \mathcal{R}(A)$ and $\mathcal{R}(P_{B^*}C^*)\subseteq \mathcal{R}(B^*)$, then Equation \eqref{e31} has the solution
$X=X_{h}+X_{p}$ and $Y=Y_{h}+Y_{p}$, where $X_{h}$ and $Y_{h}$ are defined in \eqref{xhyh}, and $X_p$ and $Y_p$ are solutions of $AX=P_{A}CN_{B}$
and $B^{*}Y^{*}=P_{B^{*}}C^{*}$, respectively.\\
 Furthermore, if Equation \eqref{e31} has a solution, then the solution is of the form $X$ and $Y$ above.
\end{theorem}
\begin{proof}
    By the assumption, we have $\mathcal{R}(CN_B)\subseteq \mathcal{R}(A)$. Therefore, $P_ACN_B=CN_B$. So that $(I-P_A)CN_B=0$. Thus $N_{A^*}CN_{B}=0$, which means that
    $(I-P_{A})C(I-P_{B^*})=0.$
    It follows that
     \begin{equation}\label{e32}
     P_{A}CN_B+CP_{B^*}=C.
     \end{equation}
On the other hand, since $\overline{\mathcal{R}(A)}$ is orthogonally complemented and $P_A$ is the projection onto $\overline{\mathcal{R}(A)}$, so $\mathcal{R}(P_ACN_B)\subseteq \mathcal{R}(A)$. Hence, due to Theorem \ref{t22}, the equation $AX=P_ACN_B$ has a reduced solution. Similarly, since $\overline{\mathcal{R}(B)}$ is orthogonally complemented and $\mathcal{R}(P_{B^*}C^*)\subseteq \mathcal{R}(B^*)$, according to Theorem \ref{t22},
the equation $B^*Y^*=P_{B^*}C^*$ has a reduced solution. In fact, Theorem~\ref{t22} gives the following solutions
\[
X_{p}=P_{A^{*}} A^{-1}P_{A}CN_{B} \qquad\text{and}\qquad Y_{p}
^{*}=P_{B} (B^{*})^{-1}P_{B^{*}}C^{*}.
\] for $AX=P_{A}CN_{B}$
and $B^{*}Y^{*}=P_{B^{*}}C^{*}$, respectively. From (\ref{e32}) we deduce that
\[
AX_p+Y_pB=P_{A}CN_{B}+CP_{B^{*}}=C.
\]
Therefore, the general solution of Equation \eqref{e31} is $X=X_p+X_h$ and $Y=Y_p+Y_h$. Note that, according to
Theorem \ref{t31}, $X_h$ and $Y_h$ satisfies $AX+YB=0$ with suitable $W_2$ and $W_3$ mentioned in Remark \ref{r31}.\\
Now, assume that there exist $X$ and $Y$ such that $AX+YB=C$, so $AXN_B=CN_B$.
Since $\mathcal{R}(CN_B)\subseteq \mathcal{R}(A)$,
$\overline{\mathcal{R}(N_BC^*)}\subseteq \overline{\mathcal{R}(N_B)}=\mathcal{R}(N_B)$ (since $N_B=I-P_{B^*}$ is a projection)
and $\overline{\mathcal{R}(A^*)}$ and $\overline{\mathcal{R}(N_B)}$ are orthogonally complemented,
the conditions of Theorem \ref{t23} are covered. Thus the reduced solution of $AXN_B=CN_B$ is $D$ such that
 $\mathcal{R}(D)\subseteq \mathcal{N}(A)^\perp$ and $\mathcal{R}(D^*)\subseteq \mathcal{N}(N_B)^\perp=\mathcal{N}(B)$, which is denoted by $X_p$ and the general solution is
\begin{eqnarray*}
X=D+N_AV_1+V_2N_{N_B}=D+N_AV_1+V_2P_{B^*}=X_p+X_h.
\end{eqnarray*}
Note that $P_{N_B}$ is the projection on $\overline{\mathcal{R}(N_B)}$ and $\overline{\mathcal{R}(N_B)}=\overline{\mathcal{R}(I-P_{B^*})}=\overline{\mathcal{R}(B^*)}^\perp$, so
 $N_{N_B}=I-P_{N_B}=P_{B^*}$ (thus $P_{N_B}=I-N_{N_B}=N_B$).\\
Again, if there exist $X$ and $Y$ such that $AX+YB=C$, so $N_{A^*}YB=N_{A^*}C$ or $B^*Y^*N_{A^*}=C^*N_{A^*}$.
Since $\overline{\mathcal{R}(N_{A^*}C)}\subseteq \overline{\mathcal{R}(N_{A^*})}=\mathcal{R}(N_{A^*})$, $\mathcal{R}(C^*N_{A^*})\subseteq \mathcal{R}(B^*)$,
and $\overline{\mathcal{R}(N_{A^*})}$ and $\overline{\mathcal{R}(B)}$ are orthogonally complemented,
the suppositions of Theorem \ref{t23} hold. Hence, $N_{A^*}YB=N_{A^*}C$ has a reduced solution $D'$
 such that
 $\mathcal{R}(D')\subseteq \mathcal{N}(N_{A^*})^\perp=\mathcal{N}(A^*)$ and $\mathcal{R}(D'^*)\subseteq \mathcal{N}(B^*)^\perp$, which is denoted by $Y_p$ and the general solution is
\begin{eqnarray*}
Y=D'+N_{N_{A^*}}V_1+V_2N_{B^*}=D'+P_AV_1+V_2N_{B^*}=Y_p+Y_h.
\end{eqnarray*}
\end{proof}
\begin{remark}
If $A$ and $B$ are matrices, then $\mathcal{R}(A)$ and $\mathcal{R}(B^*)$ are closed and there exist the Moore-Penrose inverses of $A$ and $B^*$.
 In this case, if Equation \eqref{e31} has a solution, then
$$
N_{A^*}CN_{B}=N_{A^*}(AX+YB)N_{B}=0.
$$
Note that this condition is equivalent to \eqref{Q1}, i.e. $(I-AA^{\dag})C(I-B^{\dag}B)=0$, which is consistent with \cite[Theorem 1]{BaksalaryKala1}, by putting projections $P_{A}=AA^{\dag}$ and $P_{B^*}=B^{\dag}B$.
\end{remark}
Now we consider the equation $AX+BY=C$. Fillmore and Williams proved an interesting result \cite[Corollary 1]{FillmoreWilliams} by employing the well-known Douglas theorem. The next theorem is a slight generalization of \cite[Corollary 1]{FillmoreWilliams} when the considered ranges are not closed but their norm-completions are orthogonally complemented.

\begin{theorem}\label{t15}
 Let $A, B, C\in \mathcal{L}(\mathscr{E})$. Suppose that $\overline{\mathcal{R}(A^*)}$ and $\overline{\mathcal{R}(B^*)}$
 are orthogonally complemented in $\mathscr{E}$. If $A^*B=0$, then the following statements are equivalent:
\begin{enumerate}
 \item $\mathcal{R}(C)\subseteq \mathcal{R}(A)+\mathcal{R}(B);$
\item There exist $X,Y\in \mathcal{L}(\mathscr{E})$ such that $AX+BY=C$, i.e., $AX+BY=C$ has a solution.
 \end{enumerate}
and if one of above statements holds, then
 $$CC^*\leq \lambda (AA^*+BB^*)~~ for~ some~ \lambda> 0.$$
\end{theorem}

\begin{proof}
$(1) \Rightarrow (2)$\\
Let $T=\left[
 \begin{array}{cc}
 A & B \\
 0 & 0 \\
 \end{array}
 \right]$,
$S= \left[
 \begin{array}{cc}
 C & 0 \\
 0 & 0 \\
 \end{array}
 \right]$ and
 $\hat{X}= \left[
 \begin{array}{cc}
 X & W \\
 Y & Z \\
 \end{array}
 \right]$
be operators on $\mathscr{E}\oplus \mathscr{E}$. Since $A^*B=0$, we have $T^*T=\left [
 \begin{array}{cc}
 A^*A & 0\\
 0 & B^*B \\
 \end{array}
 \right]$. Hence $\overline{\mathcal{R}(T^*T)}=\overline{\mathcal{R}(A^*A)} \oplus \overline{\mathcal{R}(B^*B)}$. By \cite[Proposition 3.7]{Lance}, $\overline{\mathcal{R}(S^*S)}=\overline{\mathcal{R}(S^*)}$ for any $S\in \mathcal{L}(\mathscr{E})$. We therefore have
 \begin{equation}\label{3.5}
 \overline{\mathcal{R}(T^*)}=\overline{\mathcal{R}(A^*)} \oplus \overline{\mathcal{R}(B^*)}.
 \end{equation}
 Since $\overline{\mathcal{R}(A^*)}$ and $\overline{\mathcal{R}(B^*)}$ are orthogonally complemented, so $\overline{\mathcal{R}(T^*)}$
 is orthogonally complemented. Now we consider $T\hat{X}=S$. Due to Theorem \ref{t22}, the equation $T\hat{X}=S$ has a reduced solution if $\mathcal{R}(S)\subseteq \mathcal{R}(T)$ and $\overline{\mathcal{R}(T^*)}$ is orthogonally complemented.
 Since $\overline{\mathcal{R}(T^*)}$ is orthogonally complemented, and by $(1)$, $\mathcal{R}(S)\subseteq \mathcal{R}(T)$. Hence the equation $T\hat{X}=S$ has a reduced solution as follows:\\
 \begin{equation}\label{3.6}
 \hat{X}=P_{T^*}T^{-1}S, \hspace{0.5cm} \hat{X}^*|_{\mathcal{N}(T)}=0,
 \end{equation}
where $P_{T^*}$ is the projection onto $\overline{\mathcal{R}(T^*)}$.\\ It follows from \eqref{3.5} that
\begin{equation}
P_{T^*}=\left[
 \begin{array}{cc}
 P_{A^*} & 0 \\
 0 & P_{B^*} \\
 \end{array}
 \right]
 \end{equation}
From \eqref{3.6}, we conclude that $W=Z=0$. Therefore, $AX+BY=C$ has a nontrivial solution.\\
\\$(2) \Rightarrow (1)$ is clear.\\
 Now we assume that one of the above statements holds. It follows from by Theorem \ref{t22}that
 $$CC^* \leq \lambda (AA^*+BB^*)$$
for some $\lambda>0$.
\end{proof}
The equation $AXA^{*}+BYB^{*}=C$ is a generalization of the equation $AXA^*=C$ that has been studied for matrices in \cite{ChangWang, DengHu, XuWeiZheng} .
Using Moore-Penrose inverse, Farid et al. \cite {FaridMoslehianWangWu} obtained its Hermitian solution in Hilbert $C^*$-module setup. The next theorem provides
a sufficient condition for a nonzero solution of $AXA^*+BYB^*=0$.
\begin{theorem}\label{T36}
Let $A, B, C\in \mathcal{L}(\mathscr{E})$. Let $\overline{\mathcal{R}(A^*)}$ and $\overline{\mathcal{R}(B^*)}$ be orthogonally complemented. If there exist $V_{1}, V_{2}$ and $V_{3}$ such that $\mathcal{R}(BV_{1}P_{A^*}+V_{2}N_{A})\subseteq \mathcal{R}(A)$ and $\mathcal{R}(V_{3}^*N_{B}-AV_{1}^*P_{B^*})\subseteq \mathcal{R}(B)$, then the equation $AXA^*+BYB^*=0$ has nonzero solutions for $X$ and $Y$.
\end{theorem}

\begin{proof}
Set $AX=BV_{1}P_{A^*}+V_{2}N_{A}$ and $YB^*=N_{B}V_{3}-P_{B^*}V_{1}A^*$. Then
\begin{equation}
(BV_{1}P_{A^*}+V_{2}N_{A})A^*+B(N_{B}V_{3}-P_{B^*}V_{1}A^*)=0.
 \end{equation}
It is therefore enough to solve $AX=BV_{1}P_{A^*}+V_{2}N_{A}$ and $YB^*=N_{B}V_{3}-P_{B^*}V_{1}A^*$.
Since $\overline{\mathcal{R}(A^*)}$ is orthogonally complemented and $\mathcal{R}(BV_{1}P_{A^*}+V_{2}N_{A})\subseteq \mathcal{R}(A)$, by utilizing true portion of \cite[Theorem 1.1]{FangYuYao}, we conclude that the equation $AX=BV_{1}P_{A^*}+V_{2}N_{A}$ has a reduced solution $X$. Similarly, since $\overline{\mathcal{R}(B^*)}$ is orthogonally complemented and $\mathcal{R}(V_{3}^*N_{B}-AV_{1}^*P_{B^*})\subseteq \mathcal{R}(B)$, so $YB^*=N_{B}V_{3}-P_{B^*}V_{1}A^*$ has a reduced solution $Y$.
\end{proof}

\begin{corollary}
Let $A, B, C\in \mathcal{L}(\mathscr{E})$, and $\overline{\mathcal{R}(A^*)}$ and $\overline{\mathcal{R}(B^*)}$ be orthogonally complemented. If $\mathcal{R}(B)=\mathcal{R}(A)$, then $AXA^*+BYB^*=0$ has a nonzero solution.
\end{corollary}
\begin{proof}
It is sufficient to put $V_2=A$, $V_3=B^*$ and consider $V_1$ as an arbitrary operator in Theorem \ref{T36}.
\end{proof}

The next theorem yields a solution of the equation $AXA^*+BYB^*=C$.
\begin{theorem}\label{t37}
Let $A, B, C\in \mathcal{L}(\mathscr{E})$. Suppose that $\overline{\mathcal{R}(A^*)}$, $\overline{\mathcal{R}(A)}$, $\overline{\mathcal{R}(B^*)}$ and $\overline{\mathcal{R}(B)}$ are orthogonally complemented. Suppose that $\mathcal{R}(C)\subseteq {\mathcal{R}(B)}$, $\mathcal{R}(C^*)\subseteq {\mathcal{R}(A)}$ and
$\mathcal{R}(C^*P_A)\subseteq {\mathcal{N}(B^*)}$. The equation $AXA^*+BYB^*=C$ has a solution if and only if $\mathcal{R}(CN_{B^*})\subseteq\mathcal{R}(A)$ and
 $\mathcal{R}(C^*N_{A^*})\subseteq \mathcal{R}(B)$.
\end{theorem}
\begin{proof}
 Note that $P_BC^*P_A=0$, since
$\mathcal{R}(C^*P_A)\subseteq {\mathcal{N}(B^*)}$.

($\Longrightarrow$) We assume that there exist $X$ and $Y$ such that $AXA^*+BYB^*=C$. So $AX^*A^*+BY^*B^*=C^*$.
By postmultiplying both sides of equation $AX^*A^*+BY^*B^*=C^*$ by $N_{A^*}$, we get
 \begin{equation}
AX^*A^* N_{A^*}+BY^*B^*N_{A^*}=C^*N_{A^*},
 \end{equation}
which gives rise to $BY^*B^*N_{A^*}=C^*N_{A^*}$. Hence $\mathcal{R}(C^*N_{A^*})\subseteq \mathcal{R}(B)$.
A similar argument yields that
\begin{equation}
AYA^*N_{B^*}=CN_{B^*},
 \end{equation}
whence $\mathcal{R}(CN_{B^*})\subseteq \mathcal{R}(A)$.

($\Longleftarrow$) By taking $\hat{X}=XA^*$ and $\hat{Y}=BY$, the equation $AXA^*+BYB^*=C$ is simplified to $A\hat{X}+\hat{Y}B^*=C$. We consider the latter equation.
We have
$$\mathcal{R}(P_BC^*)=\mathcal{R}(P_BC^*(N_{A^*}+P_A))=\mathcal{R}(P_BC^*N_{A^*}+P_BC^*P_A)=\mathcal{R}(P_BC^*N_{A^*}).$$

On the other hand, we have\\
\begin{align*}
\mathcal{R}(B)&\supseteq\mathcal{R}(C^*N_{A^*})=\mathcal{R}((P_B+N_{B^*})C^*N_{A^*})\\
&=\mathcal{R}(P_BC^*N_{A^*}+N_{B^*}C^*N_{A^*})=\mathcal{R}(P_BC^*N_{A^*}). \end{align*}
\\Thus $\mathcal{R}(P_BC^*)\subseteq \mathcal{R}(B)$. Note that $N_{B^*}C^*N_{A^*}=0$, because $\mathcal{R}(C^*N_{A^*})\subseteq \mathcal{R}(B)\subseteq\overline{\mathcal{R}(B)}=\mathcal{N}(N_{B^*})$.\\
 Since $\mathcal{R}(CN_{B^*})\subseteq \mathcal{R}(A)$ and $\mathcal{R}(P_BC^*)\subseteq \mathcal{R}(B)$, hence the assumptions of Theorem \ref{t32} are fulfilled. Hence, $\hat{X}=P_{A^*}A^{-1}P_{A}CN_{B^*}$ and $\hat{Y^*}=P_{B^*}B^{-1}P_BC^*$ is a solution of $A\hat{X}+\hat{Y}B^*=C$. Thus we reach the equations
\begin{equation}\label{sec4eq3}
XA^*=P_{A^*}A^{-1}P_{A}CN_{B^*},
\end{equation}
and
\begin{equation}\label{sec4eq4}
Y^*B^*=P_{B^*}B^{-1}P_BC^*.
\end{equation}
Since $\mathcal{R}(C^*)\subseteq {\mathcal{R}(A)}$, hence we can apply Theorem \ref{t23} to Equation \eqref{sec4eq3}.
Similarly, since $\mathcal{R}(C)\subseteq \mathcal{R}(B)$, Equation \eqref{sec4eq4} has a reduced solution according to Theorem \ref{t23}.
\end{proof}

\begin{lemma}\label{l13}
 Let $A, B\in \mathcal{L}(\mathscr{E}), T=\left[
 \begin{array}{cc}
 A & -B \\
 0 & 0 \\
 \end{array}
 \right]$ and $S=\left[
 \begin{array}{cc}
 A & B \\
 0 & 0 \\
 \end{array}
 \right]$. Suppose that $ \overline{\mathcal{R}(T^*)}$ and $\overline{\mathcal{R}(S^*)}$ are orthogonally complemented and $P$ is the orthogonal projection of $\mathscr{E}\oplus\mathscr{E}$ onto $\mathcal{N}(T)$. If $P\mathcal{N}(S)\subseteq\mathcal{N}(S)$ and $\overline{\mathcal{R}(SPS^*)^{\frac{1}{2}}}$ is orthogonally complemented
 then there exist operators $X, Y$ and $Z$ in $\mathcal{L}(\mathscr{E})$ with $X\geq0$ and $Y\geq0$ such that
 $$ \mathcal{R}(A)\cap \mathcal{R}(B)=\mathcal{R}(AX)+\mathcal{R}(AZ^*)=\mathcal{R}(BZ)+\mathcal{R}(BY)$$
 and
$$\mathcal{R}(A)\cap \mathcal{R}(B)\supseteq\mathcal{R}((AXA^*)^\frac{1}{2}),~ \mathcal{R}(A)\cap \mathcal{R}(B)\supseteq\mathcal{R}((BYB^*)^\frac{1}{2})\,.$$
\end{lemma}
\begin{proof}

 By hypothesis $ \overline{\mathcal{R}(T^*)}$ is orthogonally complemented.
 Hence $\mathcal{N}(T)$ is orthogonally complemented.
 Let $P=\left[
 \begin{array}{cc}
 X & Z^* \\
 Z & Y \\
 \end{array}
 \right]$ be the orthogonal projection of $\mathscr{E}\oplus\mathscr{E}$ onto $\mathcal{N}(T)$. Since $TP=0$, we have
 $AX=BZ$ and $AZ^*=BY$. Therefore
 $$\mathcal{R}(A)\cap\mathcal{R}(B)\subseteq\mathcal{R}(AX)+\mathcal{R}(AZ^*).$$
Equality holds here. In fact, if $w\in\mathcal{R}(A)\cap\mathcal{R}(B)$, then $w=Au=Bv$. Hence
$( u,v) \in \mathcal{N}(T)$. It follows that $u=Xu+Z^*v$, so
$$w=Au\in \mathcal{R}(AX)+\mathcal{R}(AZ^*).$$
Now let $W=\left[
 \begin{array}{cc}
 AX & AZ^* \\
 0 & 0\\
 \end{array}
 \right]$. Then $W=\frac{1}{2}SP$. We show that $\overline{\mathcal{R}(W^*)}$ is orthogonally complemented. To this end, let $x\in \mathscr{E}$. Then there are $y_1\in \overline{\mathcal{R}(S^*)} $ and $y_2\in \mathcal{N}(S)$ such that $x=y_1+y_2$. This gives $Px=Py_1+Py_2$ and so $x=Py_1+Py_2+(I-P)x$, observe that $Py_1\in \overline{\mathcal{R}(PS^*)}$ and $Py_2+(I-P)x\in \mathcal{N}(SP)$ by virtue of $P\mathcal{N}(S)\subseteq\mathcal{N}(S)$.\\
 We have
 \begin{align*}
 \mathcal{R}(BZ )+\mathcal{R}(BY)\oplus 0&=\mathcal{R}(AX)+\mathcal{R}(AZ^*)\oplus 0\\
 & \qquad \qquad (\mbox{by~} X^2+Z^*Z = X \mbox{~(deduced~from~} P^2= P))\\
 &=\mathcal{R}(W)\\
 &=\mathcal{R}(|W^*|) \qquad \qquad\qquad \quad \qquad (\mbox{by Corollary} \ref{c32}) \\
 &=\mathcal{R}\left(\left[
 \begin{array}{cc}
 (AXA^*)^{\frac{1}{2}} & 0\\
 0 & 0\\
 \end{array}
 \right]\right)\\
 &=\mathcal{R}(AXA^*)^{\frac{1}{2}}\oplus 0
 \end{align*}
The other relations can be proved in a similar manner.
\end{proof}

Employing the above lemma, we discuss the existence of nontrivial solutions for the equation $AXA^*+BYB^*=CZ$ having three unknown operators.
\begin{theorem}\label{l14}
 Let $A, B, C\in \mathcal{L}(\mathscr{E}), T=\left[
 \begin{array}{cc}
 A & -B \\
 0 & 0 \\
 \end{array}
 \right]$ and $S=\left[
 \begin{array}{cc}
 A & B \\
 0 & 0 \\
 \end{array}
 \right]$. Suppose that $ \overline{\mathcal{R}(T^*)}$ and $\overline{\mathcal{R}(S^*)}$ are orthogonally complemented and $P$ is the orthogonal projection of $\mathscr{E}\oplus\mathscr{E}$ onto $\mathcal{N}(T)$. If $P\mathcal{N}(S)\subseteq\mathcal{N}(S)$ and $\overline{\mathcal{R}(SPS^*)^{\frac{1}{2}}}$ is orthogonally complemented and $ 0\neq \mathcal{R}(A)\cap \mathcal{R}(B)\subseteq \mathcal{R}(C)$, then there exist non-zero operators $X,Y\geq0$ in $\mathcal{L}(\mathscr{E})$ and a non-zero solution $Z $ in $\mathcal{L}(\mathscr{E})$ such that
\begin{equation}
AXA^*+BYB^*=CZ.
\end{equation}
\end{theorem}
\begin{proof}
 It follows from Lemma \ref{l13} that there are $X,Y\geq0$ in $\mathcal{L}(\mathscr{E})$ such that $\mathcal{R}(A)\cap \mathcal{R}(B)\supseteq\mathcal{R}((AXA^*)^\frac{1}{2})$ and $\mathcal{R}(A)\cap \mathcal{R}(B)\supseteq\mathcal{R}((BYB^*)^\frac{1}{2})$. Since $\mathcal{R}(A)\cap \mathcal{R}(B)\subseteq \mathcal{R}(C)$, we have $\mathcal{R}((AXA^*)^\frac{1}{2})\subseteq \mathcal{R}(C) $ and $\mathcal{R}((BYB^*)^\frac{1}{2})\subseteq \mathcal{R}(C)$. Noting to $\mathcal{R}(AXA^*)\subseteq\mathcal{R}((AXA^*)^\frac{1}{2})$ and $\mathcal{R}(BYB^*)\subseteq \mathcal{R}((BYB^*)^\frac{1}{2})$, we obtain
\begin{equation}
\mathcal{R}(AXA^*+BYB^*)\subseteq \mathcal{R}(AXA^*)+\mathcal{R}(BYB^*) \subseteq \mathcal{R}(C).
\end{equation}
Employing Theorem \ref{t22}, it is concluded that there exists a reduced solution $Z$ such that $AXA^*+BYB^*=CZ$. The assumption $\mathcal{R}(A)\cap \mathcal{R}(B)\neq 0$ yields that $X$ and $Y$ are not zero by Lemma \ref{l13}. It follows from Theorem \ref{t22} that the reduced solution $Z$ is not zero too.
\end{proof}



\begin{thebibliography}{99}

\bibitem{BaksalaryKala1} J.K. Baksalary and R. Kala,
\textit{The matrix equation $AX-YB=C$},
Linear Algebra Appl. \textbf{25} (1979), 41--43.

\bibitem{BaksalaryKala2} J.K. Baksalary and R. Kala,
\textit{The matrix equation $AXB-CYD=E$},
Linear Algebra Appl. \textbf{30} (1980), 141--147.

\bibitem{ChangWang} X.W. Chang and J. Wang,
\textit{The symmetric solutions of the matrix equations $AX + YA = C$, $AXA^{T} + BYB^{T} = C$ and
$(A^T XA, B^T XB) = (C, D)$},
Linear Algebra Appl. \textbf{179} (1993), 171--189.

\bibitem{Douglas} R.G. Douglas,
\textit{On majorization, factorization and range inclusion of operators on Hilbert space},
Proc. Amer. Math. Soc. \textbf{17} (1966), 413--416.

\bibitem{DengHu} Y. Deng and X. Hu,
\textit{On solutions of matrix equation $AXA^T + BYB^T = C$},
J. Comput. Math. \textbf{23} (2005), 17--26.

\bibitem{FangYuYao} X. Fang, J. Yu and H. Yao,
\textit{ Solutions to operator equations on Hilbert $C^*$-modules},
 Linear Algebra Appl. \textbf{431} (2009), 2142--2153.

\bibitem{FangYu} X. Fang, J. Yu,
\textit{Solutions to operator equations on Hilbert $C^*$-modules II},
 Integral Equatations Operator Theory \textbf{68} (2010), 23--60.

\bibitem{FaridMoslehianWangWu} F.O. Farid, M.S. Moslehian, Q.-W. Wang and Z.-C. Wu,
\textit{On the Hermitian solutions to a system of adjointable operator equations},
 Linear Algebra Appl. \textbf{437} (2012), 1854--1891.

\bibitem{FillmoreWilliams} P.A. Fillmore and J.P. Williams,
\textit{On Operator Ranges},
Advances in Math. \textbf{7} (1971), 254--281.

\bibitem{Lance} E.C. Lance,
\textit{Hilbert $C^{\ast}$-modules: A Toolkit for Operator Algebraists},
Cambridge University Press, Oxford, 1995.

\bibitem{MMM} Z. Mousavi, F. Mirzapour and M.S. Moslehian, \textit{Positive definite solutions of certain nonlinear matrix equations},
 Oper. Matrices \textbf{10} (2016), 113--126.

\bibitem{XuSheng} Q. Xu and L. Sheng,
\textit{Positive semi-definite matrices of adjointable operators on Hilbert $C^*$-modules},
Linear Algebra Appl. \textbf{428} (2008), No. 4, 992--1000.

\bibitem{XuWeiZheng} G. Xu, M. Wei and D. Zheng,
\textit{On solutions of matrix equation $AXB + CYD = F$},
Linear Algebra Appl. \textbf{279} (1998), 93--109.

\end{thebibliography}
\end{document}